\documentclass[11pt]{amsart}
  \usepackage{amsmath,amsfonts,amssymb,amsthm,fancyhdr}
\usepackage{amssymb}

\allowdisplaybreaks

\newtheorem{theorem}{Theorem}[section]
\newtheorem{lemma}[theorem]{Lemma}
\newtheorem{proposition}[theorem]{Proposition}
\newtheorem{corollary}[theorem]{Corollary}

\theoremstyle{definition}

\theoremstyle{remark}
\newtheorem{remark}[theorem]{Remark}

\numberwithin{equation}{section}

\DeclareMathOperator{\grad}{grad}
\DeclareMathOperator{\dist}{\mathrm{dist}}
\DeclareMathOperator{\Hess}{Hess}
\DeclareMathOperator{\dv}{div}

\begin{document}

\title[Three ball theorem]{On the three ball theorem for solutions\\ of the Helmholtz equation}

\author[S.~M.~Berge]{Stine Marie Berge}
\address{Department of Mathematical Sciences, Norwegian University of Science and Technology, 7491 Trondheim, Norway}
\email{stine.m.berge@ntnu.no}

\author[E.~Malinnikova]{Eugenia Malinnikova}
\address{Department of Mathematical Sciences, Norwegian University of Science and Technology, 7491 Trondheim, Norway, and}

\address{Department of Mathematics, Stanford University, Stanford, California 94305, United States}	
\email{eugeniam@stanford.edu}


\begin{abstract}
	Let $u_k$ be a solution of the Helmholtz equation with the wave number $k$,  $\Delta u_k+k^2 u_k=0$, on a small ball in either $\mathbb{R}^n$, $\mathbb{S}^n$, or $\mathbb{H}^n$. For a fixed point $p$, we define $M_{u_k}(r)=\max_{d(x,p)\le r}|u_k(x)|.$ The following three ball inequality \[M_{u_k}(2r)\le C(k,r,\alpha)M_{u_k}(r)^{\alpha}M_{u_k}(4r)^{1-\alpha}\] is well known, it  holds for some $\alpha\in (0,1)$ and $C(k,r,\alpha)>0$ independent of $u_k$. We show that the constant $C(k,r,\alpha)$ grows exponentially in $k$ (when $r$ is fixed and small). We also compare our result with the increased stability for solutions of the Cauchy problem for the Helmholtz equation on Riemannian manifolds.
\end{abstract}

\maketitle
\section{Introduction} 
In the present work we study constants in the three ball inequality for solutions of the Helmholtz equation. We begin by recalling Hadamard's celebrated three circle theorem. Let $f$ be a holomorphic function in the disk $\mathbb{D}_R=\{z\in\mathbb{C}: |z|<R\}$. Then its maximum function \[M_f(r)=\max_{\left|z\right|\le r}|f(z)|\] satisfies the convexity condition 
\begin{equation}
   \label{E:three_circle}
M_f(r_0^{\alpha}r_1^{1-\alpha})\le M_f(r_0)^\alpha M_f(r_1)^{1-\alpha},
\end{equation}
for any $r_0,r_1<R$ and $\alpha\in \left( 0,1 \right)$. The proof of \eqref{E:three_circle} is based on the fact that $\log|f|$ is a subharmonic function. Note that by the maximum principle \eqref{E:three_circle} also holds when the maximum is taken over circles.

Surprisingly, Hadamard's theorem generalizes to other classes of functions, such as solutions of second order elliptic equations and their gradients. We refer the reader to the article \cite{La63} of Landis and to the survey \cite{ARRV09}. Three spheres theorems for the gradients of harmonic functions and, more generally, harmonic differential forms can be found in \cite{Ma00}. The three ball theorem for solutions of the Helmholtz equation on Riemannian manifolds was studied in \cite{Ma13}. This has various applications, for example it was one of the tools used to estimate the Hausdorff measure of the nodal sets of Laplace eigenfunctions, see \cite{Lo18b, Lo18a}.

We consider the Helmholtz equation
\begin{equation}
	\label{hh}
	\Delta_M u_k+k^{2} u_k=0
\end{equation}
on a domain $D$ in a Riemannian manifold $\left( M,\mathbf{g} \right)$. For $D=M$ and $M$ being a closed manifold without boundary, solutions of \eqref{hh} are $L^2$-eigenfunctions of the Laplacian. 
One of the important facts for analysis on closed manifolds is the existence of an orthonormal basis for $L^{2}\left( M \right)$ consisting of eigenfunctions of the Laplacian. The classical example is the Fourier basis on the circle $\mathbb{S}^{1}$. Such an orthonormal basis can be used to solve the heat, wave, and Schr\"odinger equations on closed manifolds, under certain conditions. 

We study properties of functions that satisfy the Helmholtz equation on some geodesic ball in the manifold. Fix a point $p\in M$ and denote by $B(p,r)$ the geodesic ball of radius $r$ centered at $p$. Then for a function $u$ we define 
\[M_{u}\left( r \right)=\max_{x\in B(p,r)}|u(x)|.\]
The following doubling inequality holds for Laplace eigenfunctions on a closed manifold
\begin{equation}
	\label{df}
	M_{u_k}(2r)\le C_1e^{C_2k}M_{u_k}(r),
\end{equation}
where $C_1$ and $C_2$ are constants only depending on the Riemannian manifold $(M,\mathbf{g})$.
Inequality \eqref{df} was first shown by Donnelly and Fefferman in \cite{DF88}. Later Mangoubi \cite[Theorem 3.2]{Ma13} gave a new proof by showing the stronger local inequality
\begin{equation}\label{eq:6}
	M_{u_k}\left( 3r \right)\le C_3e^{C_4kr}M_{u_k}\left( 2r \right)^{\alpha}M_{u_k}\left( 8r \right)^{1-\alpha},
\end{equation}
for small $r$, some fixed $\alpha\in (0,1)$, and constants $C_3$ and $C_4$ only depending on the curvature. Further results on the propagation of smallness for eigenfunctions were obtained in \cite{LM}.
In this article we show that \eqref{eq:6} is sharp in the following sense: The coefficient $C_3e^{C_4kr}$ in \eqref{eq:6} cannot be replaced by a function growing subexponentially in $kr$ as $k$ grows. This is done by constructing special families of solutions of the Helmholtz equation on  Euclidean spaces, hyperbolic spaces, and the standard  spheres.

We also compare \eqref{eq:6} with the increased stability for solutions of the Cauchy problem for the Helmholtz equation studied in \cite{HI04, IK11, BNO19}. Roughly speaking, the idea is that one can estimate the solution in the interior of some convex domain from an a priori bound and an estimate of the Cauchy data on some part of the boundary. Moreover, the estimate does not depend on $k$. For solutions of the Helmholtz equation in a geodesic ball $B\left( p,R \right)$ we prove for $r<R_1<R$ that 
		\begin{equation}\label{eq:revN}
		\int_{B(p, r)}u_k^2\,\mathrm{dvol}\le C(r,R_1)\int_{B(p, R_1)\setminus B(p, r)} u_k^2\,\mathrm{dvol},
		\end{equation}
                and call \eqref{eq:revN} the \textit{reverse three ball inequality}. A more general result can be found in \cite[Section 1.3]{ALM16}, where delicate questions regarding localization of solutions of the Schr\"odinger equation are considered. We deduce \eqref{eq:revN} from a similar estimate for the $H^1$ norms where the constant does not depend on $k$. The $H^1$ estimate is proved by a Carleman-type inequality, that can be found in \cite{I17, BNO19}.

The structure of the paper is as follows.
We prove the sharpness of the three ball inequality \eqref{eq:6} in Section~\ref{sec:2}. In Section~\ref{sec:2.1} we present the argument for the Euclidean space, while the arguments for the hyperbolic space and the sphere are given in Section~\ref{sec:2.2}. We prove inequality \eqref{eq:revN} in Section~\ref{sec:3}.  Finally, we give a simple estimate for the location of the first positive zero of the Bessel functions, and collect some comparison theorems for solutions of the Sturm--Liouville equations in Appendix.

\subsection*{Acknowledgements} The authors are very grateful to the anonymous referee for useful comments. Their suggestions, in particular, substantially improved the presentation in Section ~\ref{sec:3}.


\section{The three ball inequality}\label{sec:2}
\subsection{Bessel functions and the Helmholtz equation  in $\mathbb{R}^n$}\label{sec:2.1} 
Let $J_l$ denote  the Bessel function of the first kind. We have collected some facts about the Bessel functions in Appendix \ref{A}. If $Y_m$ is an eigenfunction of the Laplace operator on the sphere $\mathbb{S}^{n-1}$ with eigenvalue $m(m+n-2)$ then
	\[u_k(r,\theta)=r^{1-n/2}J_{m+n/2-1}(kr)Y_m(\theta)\]
solves the Helmholtz equation \eqref{hh}. Moreover, any solution of \eqref{hh} in $\mathbb{R}^n$ (or in the unit ball)  can be decomposed into a series of such solutions.

In order to study the constant in the three ball inequality \eqref{eq:6} that involves the maximum function, we analyze the behavior of the Bessel functions. From now on we assume that $n=2$ for simplicity. Our results can be easily extended to all dimensions $n\ge 2$. 
\begin{lemma}\label{cor:2}
	Let $0<\gamma<\delta<1$ and set $\beta=\sqrt{1-\delta^2}$. Then there exists a constant 
	$C$, only depending on $\gamma$ and $\delta$, such that for any positive number $m$ we have
	\begin{equation}\label{eq:Jgr}
		J_m(\gamma m)<C\left(\frac{\gamma}{\delta}\right)^{\beta m} J_m(\delta m).
	\end{equation}
\end{lemma}
\begin{proof} 
	The strategy is to apply the Sturm comparison theorem, see Theorem \ref{T:sturm-comp}. We apply the theorem to the Bessel function $J_m$ solving the Bessel equation
	\[\left( xJ_m'\left( x \right) \right)'+\frac{x^{2}-m^{2}}{x}J_{m}\left( x \right)=0,\]
	and a solution of the Euler equation 
	\begin{equation}
		\label{euler}
		\left( xy'\left( x \right) \right)'+\frac{\left( \delta^{2}-1 \right)m^{2}}{x}y\left( x \right)=0.
	\end{equation}

 	Let $y$ be the solution of \eqref{euler} satisfying the initial conditions
	\[y(\gamma m)=J_m(\gamma m)\quad {\text{and}}\quad  y'(\gamma m)=J_m'(\gamma m).\]
	We know that $J_m$ is positive and increasing on $[0,m]$. The latter can be verified by using the second derivative test and  inserting the argument of the first maximum of $J_m$ into the equation
	\[x^2J_m''(x)+xJ_m'(x)+(x^2-m^2)J_m(x)=0.\] 
	Moreover, notice that for $x\in[\gamma m,\delta m]$ we have
	\[x^2-m^2\le (\delta^2-1)m^2.\]
	Hence all the conditions in the comparison theorem are satisfied and we conclude that $y\left( x \right)\le J_{m}\left( x \right)$ on $\left[ \gamma m, \delta m \right]$.
	  
	Any solution of the Euler equation \eqref{euler} is on the form
	\[y(x)=c_1x^{m\beta}+c_2x^{-m\beta}.\]
	Using that 
	\[J_m(\gamma m)=y\left( \gamma m \right)>0 \text{ and }J_m'(\gamma m)=y'\left( \gamma m \right)>0,\]
	we conclude that  $c_1>0$ and $|c_2|<c_1\gamma^{2m\beta}m^{2m\beta}$. Thus
	\[ J_m(\gamma m)=c_1(\gamma m)^{m\beta}+c_2(\gamma m)^{-m\beta}<2c_1(\gamma m)^m\beta\]
	and
	\[y(\delta m)>qc_1(\delta m)^{m\beta},\] 
	where $q=q(\gamma, \delta)>0$.
	It follows that 
	\begin{align*}
		J_{m}\left( \gamma m \right)<2c_1\left( \gamma m \right)^{m\beta}<\frac{2}{q}\left( \frac{\gamma}{\delta} \right)^{m\beta}y\left( \delta m \right)<\frac{2}{q}\left( \frac{\gamma}{\delta} \right)^{m\beta}J_{m}\left( \delta m \right). 
	\end{align*}
\end{proof}
We can now prove the main result of this section. 
\begin{theorem}
	Assume that there is an $\alpha\in(0,1)$ and a constant $C(k,r,\alpha)$ such that for any solution $u_k$ of the Helmholtz equation \eqref{hh} the following three ball inequality holds
	\begin{equation}\label{eq:7}
		M_{u_k}(2r)\le C(k,r,\alpha)M_{u_k}(r)^{\alpha}M_{u_k}(4r)^{1-\alpha}.
	\end{equation}
	Then $C(k,r,\alpha)$ grows at least exponentially in $kr$. More precisely, $C\left( k,r,\alpha \right)\ge c e^{d\alpha kr},$ where $c$ and $d$ are absolute constants.
\end{theorem}
\begin{proof}
	Consider solutions of the Helmholtz equation on the form 
	\[u_k(r,\theta)=J_m(kr)\sin(m\theta).\]
	 The maximal function then simplifies to
	\[M_{u_k}(r)=\max_{0\le x\le kr}|J_m(x)|.\] 
	We now use the fact that for  $m>0$ the maximum of $J_{m}(x)$ is attained in the interval $\left( m,m\left( 1+\varepsilon\left( m \right) \right) \right)$, where $\varepsilon\left( m \right)\to 0$ as $m\to \infty$. This is a well known result on the asymptotic of the first zero of the Bessel functions, for the convenience of the reader we include a simple proof in Appendix~\ref{A}.
	We choose $m_0$ such that $\varepsilon\left( m \right)\le 1/3$ when $m\ge m_0$. Assume first that 
	\[kr>m_1=\max\{4, 2m_0/3\}.\]
	Then given $r$ we can find $m\ge m_0$ such that \[6kr/5<m<3kr/2.\] This implies $kr<5m/6$ and $2kr>4/3m$.
	Then $M_{u_k}(4r)=M_{u_k}(2r)$ and we can reduce \eqref{eq:7} to 
	\[\left(\frac{ M_{u_k}(2r)}{M_{u_k}\left( r \right)}\right)^{\alpha}\le C(k,r,\alpha).\] 
	Set $\gamma=5/6$ and $\delta=\frac{1+\gamma}{2}=\frac{11}{12}$.
	Applying Lemma~\ref{cor:2} together with
	\[M_{u_k}(2r)>J_m(m)>J_m(\delta m)\]
	and $M_{u_k}(r)<J_m(\gamma m)$ we conclude that 
	\[C(k,r,\alpha)\ge \left(\frac{M_{u_k}(2r)}{M_{u_k}(r)}\right)^{\alpha}\ge ce^{\alpha d m},\]
	for some positive constant $d$ that can be computed. Finally, since $m>6kr/5$ we get the required estimate when $r>k^{-1}m_1$. 
			
	Now for $r\le k^{-1}m_1$ we consider the solution $u_k(r,\theta)=J_0(kr)$. Then $M_{u_k}(r)=J_0(0)$ since \[J_0(0)=\max_{x\ge 0}|J_0(x)|,\] and we conclude that $C(k,r,\alpha)\ge 1$ for any $r>0$.
	Choosing $c < e^{-\alpha d m_{1}}$ we have for all $r>0$ that \[C(k,r,\alpha) \geq c e^{\alpha d k r}.\qedhere\]  
\end{proof}
	
\subsection{Solutions of the Helmholtz equation on the sphere and hyperbolic space}\label{sec:2.2}
In this section we repeat the argument of the sharpness of the three ball inequality on the hyperbolic space and sphere. We show in particular that assumptions on the sign of the curvature do not lead to better behavior of the constant in the three ball inequality. Again, we use the spherical symmetry of the spaces and separation of variables to construct a solution of \eqref{hh} that is the product of a radial and a spherical factor. On the sphere the radial part is given by Legendre polynomials. For the hyperbolic space the radial part is also explicitly known, see \cite[p.~4222 eq.~(2.26)]{Ca94}. Once again, in our argument we only use  the differential equation for the radial part.

We define
\begin{equation}\label{sin}
	\sin_{K}\left( r \right)=
	\begin{cases}
		\frac{\sin\left( \sqrt{K}r \right)}{\sqrt{K}},    & \quad \text{when } K>0  \\
		r,                                                & \quad \text{when } K=0  \\
		\frac{\sinh\left( \sqrt{-K}r \right)}{\sqrt{-K}}, & \quad \text{when } K<0 \\
	\end{cases}.
\end{equation}
Furthermore, we  use the associated functions
$\cos_{K}\left( r \right)=\left( \sin_{K}\left( r \right) \right)',$ $\cot_{K}\left( r \right)=\frac{\cos_{K}(r)}{\sin_{K}\left( r \right)}$, and $\tan_{K}\left( r \right)=\frac{1}{\cot_{K}\left( r \right)}$. Then the Laplacian of a simply connected $n$-dimensional Riemannian manifold $(M, \mathbf{g})$ with constant sectional curvature $K$ is given in polar coordinates by
\[\Delta_M=\frac{d^{2}}{dr^{2}}+\left( n-1 \right)\cot_{K}\left( r \right)\frac{d}{dr}+\frac{1}{\sin_{K}^{2}\left( r \right)}\Delta_{\mathbb{S}^{n-1}}.\]
In this section we work in dimension two. Assume that $u_k\left( r,\theta \right)=R\left( r \right)\Theta(\theta)$ is a solution of the Helmholtz equation. Then $R\left( r \right)$ satisfies the equation
\begin{equation}
	\label{radial}
	\sin_{K}^{2}\left( r \right)\left( \frac{R''\left( r \right)+\cot_{K}\left( r \right)R'\left( r \right) }{R(r)}+k^{2} \right)=-\frac{\Delta_{\mathbb{S}^{1}}\Theta\left( \theta \right)}{\Theta\left( \theta \right)}
	=m^2.
\end{equation}

Let $\kappa=K/k^2$ and let $L_{\kappa,m}\left( \rho \right)$ be the solution of the differential equation
\begin{multline}\label{eq:radB}
	\sin^{2}_{\kappa}\left(  \rho\right)L_{\kappa,m}''\left( \rho \right)+\sin_{\kappa}\left( \rho \right)\cos_{\kappa}\left( \rho \right)L_{\kappa,m}'\left( \rho \right) \\
	+\left( \sin_{\kappa}^{2}\left( \rho \right)-m^{2} \right)L_{\kappa,m}\left( \rho \right)=0,
\end{multline}
where $L_{\kappa,m}$ is well defined at $\rho=0$ and positive on some interval $(0,\varepsilon)$. Then for $m>0$ we have $L_{\kappa,m}\left( 0 \right)=0$. Note that when $K=0$ this equation becomes the Bessel equation. Setting
$R\left( r \right)=L_{\kappa,m}\left( kr\right)$
we get a solution to \eqref{radial}.
Then \eqref{eq:radB} can be rewritten in the Sturm-Liouville form as
\begin{align}\label{eq:11}
	\left( \sin_{\kappa}\left( \rho \right) L_{\kappa, m}'\left( \rho \right) \right)'+ \frac{\sin_{\kappa}^{2}\left( \rho \right)-m ^{2}}{\sin_{\kappa}\left( \rho \right)} L_{\kappa,m}\left( \rho \right) =0. 
\end{align}
	
We begin by estimating the maximum point of $L_{\kappa,m}$ from below. Let 
\begin{equation*}
	R_{\kappa}=\begin{cases}
	\infty,&\kappa\le 0\\
	\frac{\pi}{2\sqrt{\kappa}},&\kappa> 0
	\end{cases}.
\end{equation*}
Note that for $\rho\le R_{\kappa}$ we have that $\sin_{\kappa}\left( \rho \right) $ is increasing, or equivalently that $\cos_{\kappa}\left( \rho \right)\ge 0$.
\begin{proposition}\label{cl:1}
	Let $0 < \rho^*_1<\rho^*_2<\dots<R_{\kappa}$ be the points where $L_{\kappa,m}$ attains local maximums and minimums before $R_{\kappa}$. Then $\left|L_{\kappa,m}\left( \rho^*_i \right)\right|$ is a decreasing sequence in $i$. Moreover, the first local maximum $\rho^*_1$  satisfies $\rho^*_1\ge \sin_{\kappa}^{-1}\left( m \right)$.
\end{proposition}
\begin{proof}
	At $\rho^*_1$ we have $L_{\kappa,m}'(\rho^*_1)=0$ and \eqref{eq:11} implies that
	\[\sin_{\kappa}^{2}\left( \rho^*_1 \right) L_{\kappa,m}''\left( \rho^*_1 \right)+\left( \sin_{\kappa}^{2}\left( \rho^*_1 \right)-m^{2} \right)L_{\kappa,m}\left(\rho^*_1\right) =0.\] 
	By the second derivative test it is not possible to have a maximum before $\sin_{\kappa}^{-1}\left( m\right)$, implying the lower bound for the first local extremum.
			
	The remaining part of the proposition follows from Sonin-P\'olya oscillation theorem, see Theorem \ref{T:Sonin-Polya}. The conditions in the oscillation theorem are satisfied on the interval $\left( \sin_{\kappa}^{-1}\left( m \right),R_{\kappa}\right)$ since $\sin_{\kappa}\left( \rho \right)>0$, 
	\[(\sin_{\kappa}^{2}\left( \rho \right)-m^{2})/\sin_{\kappa}\left( \rho \right)\ne 0,\]
	and \[\left( \sin_{\kappa}\left( \rho \right) \frac{\sin_{\kappa}^{2}\left( \rho \right)-m^{2}}{\sin_{\kappa}\left( \rho \right)}\right)'=2\cos_{\kappa}\left( \rho \right)\sin_{\kappa}\left( \rho \right)>0.\] 
	Thus the sequence $|L_{\kappa,m}(\rho_i^*)|$ is decreasing.
\end{proof}
	
\begin{remark}\label{re:m_0} For $m=0$, by analyzing the differential equation \eqref{eq:radB}, we see that $L'_{\kappa,0}(0)=0$. Then the proof of Proposition~\ref{cl:1} implies that $L_{\kappa,0}(\rho)$ satisfies $L_{\kappa,0}(0)\ge |L_{\kappa,0}(\rho)|$ for $\rho>0$.
\end{remark}
 Now our aim is to prove an analog of Lemma~\ref{cor:2}. The next four results show how we can control the ratio of two values of $L_{\kappa, m}$. 
	
\begin{lemma}\label{l:comp1}
	Let $\rho_2\in(0,R_\kappa)$ and $\delta\in(0,1)$ satisfy the inequality
	$\sin_\kappa (\rho_2)\leq \delta m$. Then for $\rho_1<\rho_2$ and $\beta=\sqrt{1-\delta^{2}}$ we have the bound
	\[\frac{L_{\kappa,m}(\rho_2)}{L_{\kappa,m}(\rho_1)}\ge \frac{1}{2}\left[ \left(\frac{\tan_\kappa(\rho_2/2)}{\tan_\kappa(\rho_1/2) }\right)^{\beta m}-\left(\frac{\tan_\kappa(\rho_2/2)}{\tan_\kappa(\rho_1/2) }\right)^{-\beta m}\right].\]
\end{lemma}
\begin{proof} We compare the function $L_{\kappa,m}$ to a solution of the equation
	\begin{equation}\label{eq:diff}
	(\sin_\kappa (\rho) y'(\rho))'+\frac{m^2(\delta^2-1)}{\sin_\kappa (\rho)}y(\rho)=0.
	\end{equation}
    By the assumption we have 
	\[\sin_\kappa^2(\rho)-m^2\le (\delta^2-1)m^2=-\beta^2 m^2\]
	on the interval $\left[ \rho_1,\rho_2 \right]$. Let $y$ be the solution to \eqref{eq:diff} that satisfies the initial conditions \[y(\rho_1)=L_{\kappa,m}(\rho_1) \text{ and }y'(\rho_1)=L'_{\kappa,m}(\rho_1).\] Then the comparison theorem implies that $L_{\kappa,m}(\rho_2)>y(\rho_2)$. 
	
	The explicit solution to \eqref{eq:diff} is given by
	\[y(\rho)=c_1\tan^{\beta m}_\kappa(\rho/2)+c_2\tan^{-\beta m}_{\kappa}(\rho/2).\] 
	The first maximum $\rho_1^*$ of $L_{\kappa,m}$ satisfies $\sin_\kappa(\rho^*_1) \ge m$ implying that $\rho_2<\rho_1^*$. Therefore $L_{\kappa,m}(\rho_1)>0$ and $L'_{\kappa, m}(\rho_1)>0$. Thus we have the inequality
	\[-c_1\tan_{\kappa}^{2\beta m}(\rho_1/2)<c_2<c_1\tan_{\kappa}^{2\beta m}(\rho_1/2),\quad \]
	since $(\tan_\kappa \left( \rho/2 \right))'>0$. We conclude that $c_1>0$ and similarly to Lemma~\ref{cor:2} we get
	\[L_{\kappa, m}(\rho_2)>y(\rho_2)>c_1(\tan^{\beta m}_\kappa(\rho_2/2)-\tan_\kappa^{2\beta m}(\rho_1/2)\tan^{-\beta m}_\kappa(\rho_2/2)).\]
	The estimate of $c_2$ from below implies that
	\[L_{\kappa,m}(\rho_1)=y(\rho_1)<2c_1\tan_\kappa^{\beta m}(\rho_1/2).\]
	Combining the last two inequalities gives the result.
\end{proof} 
\begin{corollary} Suppose that $K>0$ and that $\rho_1<\rho_2<\min\{R_{\kappa}, m\delta\}$ for some $\delta\in(0,1)$. For $\beta=\sqrt{1-\delta^2}$ we have the estimate
	\begin{equation}\label{eq:comp1+}
		\frac{L_{\kappa,m}(\rho_2)}{L_{\kappa,m}(\rho_1)}\ge \frac12\left[\left(\frac{\rho_2}{\rho_1}\right)^{\beta m}-\left(\frac{\rho_2}{\rho_1}\right)^{-\beta m}\right]. 
	\end{equation}
\end{corollary}
	
\begin{proof}
	We note that $\sin_\kappa(\rho_2)<\rho_2<m\delta$.
	Applying Lemma~\ref{l:comp1} and using the elementary inequality $b\tan x\ge \tan bx$ for $b\in(0,1)$,  the result follows since 
	\[\frac{\tan_\kappa(\rho_2/2)}{\tan_\kappa(\rho_1/2)}=\frac{\tan (\sqrt{\kappa}\rho_2/2)}{\tan(\sqrt{\kappa}\rho_1/2)} \ge \frac{\rho_2}{\rho_1}.\qedhere \]
\end{proof}
	
\begin{corollary}
	Let $K<0$ and suppose that \[\rho_1<\rho_2< \min\{R_{|\kappa|}, 2m\delta/3\}\] for some $\delta\in(0,1)$. Then for $\beta=\sqrt{1-\delta^2}$ and $A = \sin_\kappa \left( \rho_2 \right)/\rho_2$ we have
	\begin{equation}\label{eq:comp1-}
		\frac{L_{\kappa,m}(\rho_2)}{L_{\kappa,m}(\rho_1)}\ge \frac12\left[\left(\frac{\rho_2}{A\rho_1}\right)^{\beta m}-\left(\frac{\rho_2}{A\rho_1}\right)^{-\beta m}\right]. 
	\end{equation}
\end{corollary}	
\begin{proof} Since $\sqrt{|\kappa|} \rho<\pi/2$ and $\sinh$ is convex we have 
	\[\sin_{\kappa}\left( \rho_2 \right)\le 2\rho_2\sinh\left( \pi/2 \right)/\pi<3\rho_2/2<m\delta.\]	
	Applying Lemma~\ref{l:comp1} together with 
	\[(\log(\tanh x))'\ge\frac{ \rho_2\sqrt{-\kappa}}{\sinh \left( \rho_2\sqrt{-\kappa} \right)x} \text{ for } x<\rho_2\sqrt{-\kappa}/2,\]
	 gives \eqref{eq:comp1-}, since
	\begin{equation*}
		\frac{\tan_\kappa(\rho_2/2)}{\tan_\kappa(\rho_1/2)}=\frac{\tanh(\sqrt{|\kappa|}\rho_2/2)}{\tanh(\sqrt{|\kappa|}\rho_1/2)}\ge \frac{\rho_2}{A\rho_1}.\qedhere
	\end{equation*}
\end{proof}
We want to estimate the ratio of the values of $L_{\kappa,m}$ at two points $\rho_2>\rho_1>\sin_\kappa^{-1}(m)$. In contrast with the Bessel functions, we do not locate the maximum precisely.
					
    \begin{lemma}\label{lem:up}
		Suppose that $0<\rho_1<R_{|\kappa|}$ and $\sin_\kappa (\rho_1)>\xi m$, where $\xi>1$. There is an absolute constant $C > 0$ such that
		\begin{equation}\label{eq:comp2}
			\frac{\max_{\rho} |L_{\kappa,m}(\rho)|}{\max_{\rho\le \rho_1}|L_{\kappa,m}(\rho)|}\le 1+\frac{C}{(\xi-1)m}.
		\end{equation}
	\end{lemma}
					
	\begin{proof}
		Let $\rho^*_1$ be the first local maximum of $L_{\kappa,m}$.
		By Proposition~\ref{cl:1} if $\rho_1>\rho^*_1$ then the left-hand side of \eqref{eq:comp2} is one and the statement becomes trivial.
							
		The rest of the proof  relies  on the comparison of $L_{\kappa,m}$ and  a solution of the equation 
		\begin{equation}
			\label{eq:diff2}
			(\sin_\kappa (\rho) y')'+\frac{m^2(\xi^2-1)}{\sin_\kappa (\rho)}y=0
		\end{equation}
		on the interval $(\rho_1,\infty)$. Solutions to \eqref{eq:diff2} are of the form
		\[y(\rho)=c_1\cos(\gamma\log(\tan_\kappa(\rho/2)))+c_2\sin(\gamma\log(\tan_\kappa(\rho/2))),\]
		where $\gamma^2=\xi^2-1$.
		Let $d=\gamma\log(\tan_\kappa(\rho_1/2))$ and choose a solution $y$ of \eqref{eq:diff2} on the form
		\[y(\rho)=C_1\cos(\gamma\log(\tan_\kappa(\rho/2))-d)+C_2\sin(\gamma\log(\tan_\kappa(\rho/2))-d),\]
		with initial data $y(\rho_1)=L_{\kappa,m}(\rho_1)$ and $y'(\rho_1)=L'_{\kappa,m}(\rho_1)$. This gives the values
                \[C_1=L_{\kappa,m}(\rho_1),\quad C_2=\frac{L'_{\kappa,m}(\rho_1)\sin_\kappa(\rho_1)}{\gamma}.\]
		Applying the comparison theorem, we get
		\[L_{\kappa, m}(\rho_2)\le C_1\cos(\gamma\log(\tan_\kappa(\rho_2/2))-d)+C_2\sin(\gamma\log(\tan_\kappa(\rho_2/2))-d).\]
						
		Since $\sin$ and $\cos$ are bounded by $1$, we estimate $C_2/L_{\kappa,m}(\rho_1)$ from  above to prove \eqref{eq:comp2}. In order to  estimate $C_2$, we see that the assumption $\rho_1 < \rho_{1}^{*}$ implies  $L_{\kappa,m}(\rho) > 0$ and $L_{\kappa,m}'(\rho) > 0$ on the interval $(0, \rho_1)$. Equation \eqref{eq:11} shows that $L_{\kappa,m}''(\rho)<0$ when $\rho\in(\rho_0,\rho_1)$, where $\rho_0=\sin_\kappa^{-1}(m)$. Then the Taylor formula gives
		\[L_{\kappa,m}(\rho_0)-L_{\kappa,m}(\rho_1)+(\rho_1-\rho_0)L'_{\kappa,m}(\rho_1)<0.\]
		Consequently,
		\[L_{\kappa,m}'(\rho_1)<\frac{L_{\kappa,m}(\rho_1)-L_{\kappa,m}(\rho_0)}{\rho_1-\rho_0}<\frac{L_{\kappa,m}(\rho_1)}{\rho_1-\rho_0}.\]

			For $K>0$ the inequality 
			\[x_1-x_0>\sin x_1-\sin x_0\quad\text{when} \quad x_1>x_0\]
			 implies that $\rho_1-\rho_0>(\xi-1)m$. 
			 For $K<0$, we note that 
			\[x_1-c\sinh x_1>x_0-c\sinh x_0\quad \text{when}\quad x_1>x_0,\] 
			if $\cosh x_1<c^{-1}$. Using the assumption $\rho_1<R_{|\kappa|}$, we conclude that $\rho_1-\rho_0>c(\xi-1)m$, where $c=(\cosh \pi/2)^{-1}$. 
			Finally, we obtain 
			\[\max_{\rho\ge \rho_1}|L_{\kappa,m}(\rho)|\le \sqrt{C_1^2+C_2^2}\le L_{\kappa,m}(\rho_1)\left(1+\frac{C}{(\xi-1)m}\right).\qedhere\]
		\end{proof}
					
		Now we are ready to prove that the coefficient in the three ball theorem grows exponentially in $rk$ if we restrict ourselves to balls with sufficiently small radius $r$.
					
		\begin{theorem}
			Let $(M, \mathbf{g})$ be either a hyperbolic plane or a sphere and denote its curvature by $K$. Suppose that for some $\alpha \in \left( 0,1 \right)$ there exists a constant $C_\alpha(k,r,K)$ such that for any solution $u_k$ to the Helmholtz equation \eqref{hh}
			the following inequality holds
			\[M_{u_{k}}\left( 2r \right)\le C_\alpha(k,r,K)M_{u_{k}}\left( r \right)^\alpha M_{u_{k}}\left( 4r \right)^{1-\alpha},\quad 0 < r <\frac{\pi}{8\sqrt{|K|}}.\]
			Then
			\begin{equation}
				\label{eq:main-alpha}
				C_\alpha(k,r,K)\ge c_1^\alpha e^{c_2\alpha kr},
			\end{equation}
			where $c_1$ and $c_2$ only depend on $K$.
		\end{theorem}
					
		\begin{proof} Consider the family of functions
			\[u_{k,m}(r,\theta)=L_{\kappa,m}(kr)\sin\left( m\theta \right),\]
			where $m$ is a non-negative integer. By construction, $u_{k,m}$ solves the Helmholtz equation. Thus for any $m$ we have the inequality
			\[C_\alpha(k,r,K)\ge \left(\frac{M_m(2kr)}{M_m(kr)}\right)^\alpha\left(\frac{M_m(2kr)}{M_m(4kr)}\right)^{1-\alpha},\]
			where 
			\[M_m(\rho)=\max_{x\le \rho}|L_{\kappa,m}(x)|.\]
			Note that choosing $m=0$ gives $C_\alpha(k,r,K)\ge 1$ by Remark~\ref{re:m_0}. Thus if we assume that $kr<C_1$ for some constant $C_1$, we may choose $c_2$ and $c_1$ small enough such that the inequality holds.
							
			Assume first that $K<0$ so that $(M,\mathbf{g})$ is the hyperbolic plane. If $kr>C_1$ we choose a positive integer $m$ such that $10m<18kr<11m$. We apply \eqref{eq:comp1-} with $\rho_1=kr$ and $\rho_2=\mu kr<2kr$, where $\mu=19/17$. Then $\rho_2<2/3m\delta$ with $\delta<1$. We obtain
			\[
				\frac{M_m(2kr)}{M_m(kr)}\ge\frac{L_{\kappa,m}(\rho_2)}{L_{\kappa,m}(\rho_1)}\ge \frac{1}{2}\left[\left(\frac{\rho_2}{A\rho_1}\right)^{\beta m}-\left(\frac{\rho_2}{A\rho_1}\right)^{-\beta m}\right],
			\]
			where $\beta = \sqrt{1 - \delta^2}$ and \[A=\sin_\kappa\left( \rho_2 \right)/\rho_2<\sin_\kappa\left( 2kr \right)/(2kr)<4\sinh(\pi/4)/\pi<10/9.\]
			Therefore $q=\mu/A>1$ is an absolute constant and we have
			\begin{equation}\label{eq:Mratio}
				\frac{M_m(2kr)}{M_m(kr)}\ge \frac{1}{2}(q^{\beta m}-q^{-\beta m}).
			\end{equation}
			Thus there are $c_1>0$ and $c_2>0$ such that 
			\[M_m(2kr)\ge c_1\exp(c_2m)M_m(kr).\]
			On the other hand, we have $2kr>\xi m$ for $\xi =10/9$. Applying \eqref{eq:comp2} we get
			\[\frac{M_m(4r)}{M_m(2r)}=\frac{\max_{\rho\le 4kr}|L_{\kappa,m}(\rho)|}{\max_{\rho\le 2kr}|L_{\kappa,m}(\rho)|}\le 1+ C/m\le C_0,\]
			where $C_0$ is an absolute constant. Note also that $m\gtrsim kr$. Then \eqref{eq:main-alpha} follows for negative curvature. 							

			Assume now that $K > 0$ so that $(M,\mathbf{g})$ is a sphere. If $kr>C_1$ we choose $m$ to be a positive integer such that $10m<12 kr<11m$. We first let $\rho_1=kr$ and $\rho_2=13kr/12$ and apply \eqref{eq:comp1+} with $\delta=143/144$. Thus \eqref{eq:Mratio} follows whenever $kr > C_1$. Using \eqref{eq:comp2} with $\rho_1=2kr$ we need to check that $2\rho_1>\xi\pi m$ for some $\xi>1$. Note that $2\rho_1=4kr>10/3m$ and choose $\xi<\frac{10}{3\pi}$. Then \eqref{eq:main-alpha} follows for positive curvature.
		\end{proof}


\section{The reverse three ball inequality}\label{sec:3}
		
The question of stability of the solution to the Cauchy problem for the Helmholtz equation and the dependence of the estimates on the wave number $k$ was studied by many authors, see e.g.~\cite{HI04, SI07, IK11, BNO19}. We include a special case of the results adapted to the case of Riemannian manifolds  to demonstrate the difference between the usual three ball theorem and the reverse one.
\subsection{Reverse Inequality}
	
	 Let $(M,\mathbf{g})$ be a Riemannian manifold with sectional curvature satisfying \[\kappa \mathbf{g}(X,X)\le \sec(X,X)\le K \mathbf{g}(X,X).\] We denote by $\grad_M$ and $\Delta_M$ the gradient and Laplace operators on functions on $M$.  Let $B$ be a geodesic ball with diameter strictly less than the injectivity radius of $(M,\mathbf{g})$. Additionally, in the case that $K>0$ we assume that the diameter of $B$ is strictly less than $\frac{\pi}{2\sqrt{K}}$. 
	 
	 \begin{theorem}\label{th:last} Let $u_k$ solve the Helmholtz equation $\Delta_M u_k+k^2u_k=0$ in $B=B(p, R)$ and let $r<R_1<R$. There exists $C=C(r,R_1)$ such that
	 		\begin{equation}\label{eq:revL}
	 	\int_{B(p, r)}u_k^2\,\mathrm{dvol}\le C(r,R_1)\int_{B(p, R_1)\setminus B(p, r)} u_k^2\,\mathrm{dvol},
	 	\end{equation}
	 \end{theorem}
	 The result is very closed to a particular case of the result in \cite{BNO19}, we sketch the proof for the convenience of the reader.
	 
	 We say that a function $\phi:B\to\mathbb{R}$ is strictly convex if its Hessian is positive definite. We choose a point $x$ such that $x\not\in B(p,R)$ but $R+\dist(x,p)$ is strictly less than the injectivity radius and than $\frac{\pi}{2\sqrt{K}}$ for the case $K>0$, and consider $\phi(y)=\dist(x,y)^2$. This function is smooth on $B$ since the metric on the Riemannian manifold is assumed to be smooth and $x\not\in B$ while $B$ is contained in the ball of the injectivity radius around $x$. Moreover,  $\Hess(\phi)$ is (uniformly)  positive definite on $B$ and $\phi$ has no critical points, see \cite[Theorem 6.4.8]{pe16} and the preceding discussions. By repeating the computations of \cite[Lemma 1]{BNO19}, where it is also pointed out that the result holds on Riemannian manifolds, we obtain the following point-wise inequality. Let $w\in C^2(B)$ and let $v=e^{t\phi}w$,  then
	 \begin{multline*} e^{2t\phi}(\Delta_M w+k^2w)^2\ge 2\dv(b\grad_M v+a)+4t\langle \Hess_M(\phi)\grad_M v,\grad_M v\rangle_M\\
	 +4t^3\langle \Hess_M\phi\grad_M\phi,\grad_M \phi\rangle_M v^2+t\langle\grad_M\Delta_M\phi,\grad_M v\rangle_M v,
	 \end{multline*} 
	where $b=-tv\Delta_M\phi-2t\langle \grad_M v,\grad_M \phi\rangle_M$ and \[a=t(|\grad _M v|_M^2-(k^2+t^2|\grad_M \phi|_M^2)v^2)\grad_M\phi.\]	
		\begin{lemma}\label{lem:com}
                   Let $(M,\mathbf{g})$  and $B$ be as above. Then there exists a constant $ c_0>0$ such that for any function $w\in C^2_0(B)$ and $k\ge 0$ the following inequality holds
			\begin{equation}\label{coer}
			\int_B |\Delta_M w+k^2 w|^2\,\mathrm{dvol}\ge c_0\int_B|w|^2+|\grad_M w|_M^2\,\mathrm{dvol}.
			\end{equation}
		\end{lemma}
		
		\begin{proof}
			We repeat the argument given in \cite[Corollary 1]{BNO19}. Integrating the last inequality over a ball $B$ and taking into account that functions $a$ and $b$ have compact supports in $B$, we conclude that the divergence term disappears. For the next two terms, which contain the Hessian of $\phi$,  we use the convexity inequality
			\[\langle \Hess_M \phi\grad_M f,\grad_M f\rangle_M\ge c_\phi |\grad_M f|^2_M\]
			and the computation 
			\[\grad_M v=e^{t\phi}(\grad_M w+tw\grad_M \phi).\]
			Finally, the last term is estimated as
			\[\left|t\langle\grad_M\Delta_M\phi,\grad_M v\rangle_M v\right|\le \varepsilon t|\grad_M v|_M^2+\varepsilon^{-1}t|\grad_M\Delta_M \phi|_M^2v^2.\]
			Combining these inequalities, we get
			\[\int_B|\Delta_M w+k^2w|^2e^{2t\phi}\,\mathrm{dvol}\ge c_1\int_B\left(t^3|w|^2e^{2t\phi}+t|\grad w|^2e^{2t\phi}\right)\,\mathrm{dvol},\]
			when $t>t_0$. The powerful feature  of the last inequality is that $c_1$ and $t_0$ do not depend on $k$ (but depend on $\phi$ which we fix).  Finally, we fix some $t>t_0$ and let $M=\max_B e^{2t\phi}$ and $m=\min_Be^{2t\phi}$. Then \eqref{coer} holds with $c_0=c_1m\min\{t^3,t\}M^{-1}$.
		\end{proof}

				Suppose now that $u_k$ is a solution of the Helmholtz equation \eqref{hh} in a ball $B9p,R)$ that satisfies the conditions in Theorem~\ref{th:last}. We apply inequality \eqref{coer} to $w=u_k\chi$, where $\chi\in C^2_0(B)$ is compactly supported on $B$ and equals to one on a smaller ball $B_1\subset\subset B$. This gives the inequality
		\[\int_{B_1}\!\!|u_k|^2+|\grad u_k|^2\,\mathrm{dvol}\le\frac{1}{C_0}\int_{B\setminus B_1}\!\!\!|u_k\Delta_M \chi+2\grad u_k\cdot\grad \chi|^2\,\mathrm{dvol}.\]
		The last inequality implies that for any $r<R$ such that $B_R=B(x,R)$ and $B_r=B(x,r)$ are geodesic balls satisfying the conditions in Lemma~\ref{lem:com}, there is a constant $C_2(r,R)$ such that 
		\begin{multline}\label{rev}
			\int_{B_r}|u_k|^2+|\grad u_k|^2\,\mathrm{dvol}\le \\
			C_2(r,R)\int_{B_R\setminus B_r}|u_k|^2+|\grad u_k|^2\,\mathrm{dvol}.
		\end{multline}
		Inequality \eqref{rev} shows that if $u_k^2+|\grad u_k|^2$ is small on the annulus $B_R\setminus B_r$, then it is small on the whole ball $B_R$. For the Euclidean space an alternative proof can be obtained by decomposing a solution $u_k$ into series of products of Bessel functions and spherical harmonics. From this, one can deduce \eqref{rev} from the Debye asymptotic of the Bessel functions.
	
		To compare with the previous section and finish the proof Theorem \ref{th:last}, we can also use Caccioppoli's inequality to control the Sobolev norm of $u_k$ by its $L^2$-norm.
		\begin{lemma}[Caccioppoli's inequality]\label{Ca:ieq}
			Let $\varepsilon>0$ and let $R=R(M)$ be small enough. Furthermore, let $\varepsilon<r<R-2\varepsilon$. We  denote \[\Omega=B\left( x,R \right)\setminus B\left(x,r \right)\quad{\text{ and}}\] 
			\[\Omega_+=B(x, R+\varepsilon)\setminus B(x, r-\varepsilon),\quad \Omega_-=B(x,R-\varepsilon)\setminus B(x, r+\varepsilon).\]
			Assume that $u_k\in C^{2}\left( \Omega \right)$ and $\Delta_M u_k+ k^2u_k=0$ in $\Omega_+$. Then there exists a constant $C=C(M)$ such that
			\begin{align*}
			k^2 \int_{\Omega_-}u_k^{2}\,\mathrm{dvol}-\frac{C}{\varepsilon^{2}}\int_{\Omega}u_k^2\,\mathrm{dvol}&\le\int_{\Omega}\left|\grad u_k\right|^{2}\,\mathrm{dvol}\\
			&\le \left(k^2+\frac{C}{\varepsilon^{2}}\right) \int_{\Omega_+}u_k^{2}\,\mathrm{dvol}.
			\end{align*}
		\end{lemma}
		\begin{proof}
			There exists a smooth function  $\varphi_+$ with compact support in $\Omega_+$ that satisfy $\varphi_+=1$ on $\Omega$ and $\left|\grad \varphi_+\right|\le \frac{C}{\varepsilon}$ and $|\Delta_M \varphi_+|\le \frac{C}{\varepsilon^{2}}$. Then,  using the divergence theorem, we have 
			\begin{align*}
			k^2\int_{\Omega_+}\varphi _+u_k^2\,\mathrm{dvol}&=-\int_{\Omega_+}\varphi_+ u_k\Delta_M u_k\,\mathrm{dvol}\\  
			&=\int_{\Omega_+}\langle \grad u_k,\grad \left( \varphi_+ u_k \right)\rangle \,\mathrm{dvol}                         \\
				& =\int_{\Omega_+}\varphi_+\left|\grad u_k\right|^{2}\,\mathrm{dvol}-\frac{1}{2}\int_{\Omega_+}u_k^{2}\Delta_M \varphi_+ \,\mathrm{dvol}.
				\end{align*}
			Hence 
			\[\int_{\Omega}\left|\grad u_k\right|^{2}\,\mathrm{dvol}\le\left( k^2+C\varepsilon^{-2}\right)\int_{\Omega_+}u_k^{2}\,\mathrm{dvol}.\]
			On the other hand, choosing a similar function $\varphi_-\in C_0^\infty(\Omega)$ such that $\varphi_-=1$ on $\Omega_-$, we conclude that
			\[\int_{\Omega}|\grad u_k|^2\,\mathrm{dvol}\ge k^2\int_{\Omega_-}u_k^2-C\varepsilon^{-2}\int_{\Omega}u_k^2\,\mathrm{dvol}.\qedhere\]
		\end{proof}
		
	Finally, 	we go back to the inequality \eqref{rev}, and apply the  Caccioppoli inequality. Rename $R_1=R+\varepsilon$ and $r_1=r-\varepsilon$. This gives the following estimate of the $L^2$-norm of a solution to the Helmholtz equation by its $L^2$ norm on an annulus
		\begin{align*}
k^2	\int_{B_{r_1}}u_k^2 \,\mathrm{dvol}&\le	\int_{B_r}|\grad u_k|^2\,{\mathrm{dvol}}+C\varepsilon^{-2}\int_{B_{r}}u_k^2\,\mathrm{dvol}\\&	 \le (k^2+C\varepsilon^{-2})\int_{B_{R_1}\setminus B_{r_1}} u_k^2 \, \mathrm{dvol}+ C\varepsilon^{-2}\int_{B_{r_1}}u_k^2\,\mathrm{dvol},
		\end{align*}	
		for any $r_1<R_1$. Then for $k>C\varepsilon^{-1}$ the inequality \eqref{eq:revN} follows. For $k<C\varepsilon^{-1}$, we use \eqref{rev} and the Caccioppoli inequality again,  to see that
		\begin{equation}\label{eq:inverse_ball}
		\int_{B_r}|u_k|^2\,\mathrm{dvol}\le C(r,R)(k^2+C\varepsilon^2)\int_{B_{R_1}\setminus B_{r_1}}|u_k|^2\,\mathrm{dvol}.
		\end{equation}
		Thus, since $k<C\varepsilon^{-1}$, the inequality \eqref{eq:revN} follows also for that case. This conclude the proof of Theorem \ref{th:last}.

		Additionally, one might wonder if Theorem \ref{th:last} could work for large radii? The answer is in general no. As an example consider the circle and the eigenfunctions $\mathrm{Re}(z^n)$ restricted to the sphere. This sequence of eigenvalues concentrate on the equator. Hence if you center the balls at the north-pole of the sphere, and the annulus on the southern hemisphere not intersection the equator the conclusion of Theorem \ref{th:last} cannot work for the family $\mathrm{Re}(z^n)$.

\subsection{Comparison With Other Results}
	As briefly mentioned in the introduction, \cite{ALM16} contains a result for the Schr\"odinger equation with Dirichlet boundary conditions on the disk
	\begin{equation}\label{eq:dirichlet}
		\begin{cases}
			(\Delta  +V(x))u(x)=0& \text{in }\mathbb{D}\\ 
			u=0 & \text{on }\partial\mathbb{D}
	\end{cases}
	\end{equation}
	with smooth potential $V\in C^\infty(\overline{\mathbb{D}})$.  Assume that $\Omega$ is an open set in $\overline{\mathbb{D}}$ intersecting the boundary, i.e.\ $\Omega\cap \partial \mathbb{D}\neq \emptyset$.
	The following inequality holds 
	\[\int_\mathbb{D}u^2(x)\,dx\le C(V,\Omega)\int_\Omega u^2(x)\, dx,\]
	for $C(V,\Omega)$ depending only on $V$ and $\Omega$.
	When applied to $V=k^2$ this result and Theorem \ref{th:last} is very similar except that the choice of domain is more flexible.

	The article \cite{ALM16} also contains a result involving the boundary and the normal derivative. Let $\Gamma\subset \partial\mathbb{D}$ be a non-empty open set and let $u$ be the solution to \eqref{eq:dirichlet}. Then the result states that
	\[\int_\mathbb{D}u^2(x)+|\grad u(x)|^2\, dx\le C(V,\Gamma)\int_{\Gamma} u_n^2(x)\, dS.\]

\appendix
	\section{The first positive zero of the Bessel function}\label{A}
	
		Let $l$ be a non-negative half-integer, and let $\Gamma$ denotes the gamma function. The Bessel function $J_l$ is a solution to the second order ODE
		\begin{equation}
			\label{ODE_bessel}
			\rho^{2}J_l''\left( \rho \right)+\rho J_l'\left( \rho \right)+\left( \rho^{2}-l^{2} \right)J_l\left( \rho \right)=0,
		\end{equation}
		which is bounded at the origin and normalized by the condition
		\begin{equation*}
		\lim_{\rho\to 0}\rho^{-l}J_l(\rho)=2^{-l}\Gamma(l+1)^{-1}.
		\end{equation*}
		For alternative definitions and many useful asymptotic formulas for the Bessel functions we refer the reader to \cite{Ol74}. 
		
		It is well known that the first positive zero of $J_l$, usually denoted by $j_l$, satisfies $j_l\asymp l+cl^{1/3}$ as $l\to \infty$. We already explained in the proof of Lemma~\ref{cor:2} that $j_l>l$ and we give a simple proof of the inequality $j_l\le l+cl^{1/3}$.
		This will be done by comparing the Bessel equation to the following equation
		\begin{equation}\label{eq:com}
		\rho^2y_l''(\rho)+\rho y_l'(\rho)+(a_l^2\rho^2-1/4)y_l(\rho)=0
		\end{equation}
		on the interval $\rho\in [l+l^{1/3},+\infty)$.

		Suppose that $a_l<1$ satisfies 
		\begin{equation}\label{eq:a}
			(l+l^{1/3})^2(1-a_l^2)\ge l^2-1/4.
		\end{equation}
		Then the Sturm-Picone comparison theorem, see Theorem \ref{T:sturm-comparison}, implies that between any two zeros of $y_l$ there is a zero of $J_l$. It is easy to check that $y_l(\rho)=\rho^{-1/2}\cos(a_l\rho)$ solves \eqref{eq:com} and has roots at $(\pi/2+k\pi)/a_l$ for $k\in \mathbb{Z}$. We choose $a_l=l^{-1/3}$. Then \eqref{eq:a} holds for $l$ large enough. Hence $J_l$ has a root on the interval $[l+l^{1/3}, l+l^{1/3}+\pi l^{1/3}]$ and $j_l\le l+(\pi+1)l^{1/3}$.
					 
		\section{Comparison theorems for Sturm-Liouville equations}
		Classical Sturm-Liouville theory is concerned with second order differential equations on the form 
		\begin{equation*}
			(p(x)y'(x))'+q(x)y(x)=0 \ \text{on }[a,b].
		\end{equation*}
                Special cases of Sturm-Liouville equations are the radial solutions to the Helmholtz equation, see \eqref{radial}. To estimate solutions to  these  radial equations in Section \ref{sec:2.2}, we  compare them to some more simple Sturm-Liouville equations. To do this we use the following  classical theorems:
\begin{theorem}[Sturm-Picone Comparison Theorem, {\cite[Theorem B]{Hi05}}]\label{T:sturm-comparison}
	Let $y_1$ and $y_2$ be non-zero solutions to 
		\begin{align*}
			(p_1(x)y_1'(x))'+q_1(x)y_1(x)&=0, \\
			(p_2(x)y_2'(x))'+q_2(x)y_2(x)&=0,
		\end{align*}
		on the interval $[a,b]$.
		Assume that $0<p_2\le p_1$ and $q_1\le q_2$, and let $z_1$ and $z_2$ be two consecutive zeros of $y_1$. Then either $y_2$ has a zero in the interval $(z_1,z_2)$, or $y_1=y_2$.
\end{theorem}
\begin{theorem}[Sturm Comparison Theorem, {\cite[Chapter 13.7]{Du12}}]\label{T:sturm-comp}
	Let $y_1>0$ on $(a,c)$ and $y_2$ be non-zero solutions to 
		\begin{align*}
			(p(x)y_1'(x))'+q_1(x)y_1(x)&=0, \\
			(p(x)y_2'(x))'+q_2(x)y_2(x)&=0,
		\end{align*}
		on the interval $(a,c)\subset[a,b]$.
		Assume that $p>0$ and $q_2\le q_1$ on $[a,b]$. Furthermore, assume that 
		\begin{align*}
			y_1(a)=y_2(a)\ge 0\text{ and } y_1'(a)=y_2'(a)\ge 0.
		\end{align*}
		Then $y_1(x)<y_2(x)$ for all $x\in (a,c)$.
\end{theorem}
It is also important for us to estimate the maximum of some solutions to the Helmholtz equation. To limit the search, we use the following theorem:
\begin{theorem}[Sonin-P\'olya Oscillation Theorem, {\cite[Chapter 13.7]{Du12}}]\label{T:Sonin-Polya}
	Let $y$ be a solution of the differential equation \[(p(x)y'(x))' + q(x)y(x) = 0,\] where $p$ and $q$ are continuously differentiable functions on $[a, b]$. Suppose that $p > 0$, $q \neq 0$ and $(pq)' > 0$ on $(a, b)$. Then the successive local maximums of $|y(x)|$ form a decreasing sequence.
\end{theorem}
\bibliographystyle{abbrv}
\bibliography{three}
\end{document}